%% file: document.tex
\documentclass[a4paper]{article}

\usepackage[english]{babel}
\usepackage[utf8x]{inputenc}
\usepackage[T1]{fontenc}

\usepackage[a4paper,top=3cm,bottom=2cm,left=3cm,right=3cm,marginparwidth=1.75cm]{geometry}
\usepackage{amsmath, amsfonts, amssymb, amsthm, graphicx}
\usepackage[colorinlistoftodos]{todonotes}
\usepackage[colorlinks=true, allcolors=blue]{hyperref}

\usepackage[square, numbers,sort&compress]{natbib}
\usetikzlibrary{decorations.pathreplacing}
\usepackage{standalone}

\usepackage{caption}
\usepackage{cleveref}
\usepackage{multirow}
\usepackage{mathrsfs}
\usepackage{standalone}
\usepackage[ruled,vlined,linesnumbered]{algorithm2e}
\usepackage{enumitem}
\usepackage[toc,page]{appendix}
\usepackage{pdflscape}

\theoremstyle{plain}
\newtheorem{theo}{Theorem}
\newtheorem{prop}[theo]{Proposition}

\newtheorem{definition}[theo]{Definition}

\theoremstyle{remark}

\theoremstyle{remark}
\newtheorem{remX}{Remark} % \newtheorem establishes the object heading
\newenvironment{rem}    % this is the environment name for the input
	{%
	\pushQED{\qed}\begin{remX}}
	{\popQED\end{remX}}
\newtheorem{exX}{Example}
	{%
	\pushQED{\qed}\begin{exX}}
	{\popQED\end{exX}}

\input{commands.sty}

\title{Resource constrained shortest path algorithm for EDF short-term thermal production planning problem}
\author{Markus Kruber, Axel Parmentier, Pascal Benchimol}

\begin{document}
\maketitle

\input{abstract}
\input{intro}
\input{problem}

\input{monoid}
\input{singleunit}
% \input{subfiles/multi}
% \input{subfiles/appendix}

\subsection*{Acknowledgments}

We thank the ``Programme Gaspard Monge pour l'optimisation, la recherche opérationnelle et leurs intéractions avec les sciences des données'' (PGMO) for its support to this work.

\bibliographystyle{plainnat}
\bibliography{biblioPGMO}{}

\end{document}

%% file: abstract.tex
%!TEX root=../document.tex
% \comAP{Add everywhere \emph{multi} unit}
\begin{abstract}
Unit commitment problem on an electricity network consists in choosing the
production plan of the plants (units) of a
company % firm
in order to meet demand constraints. It
is generally solved using a decomposition approach where demand
constraints are
% \todoMK{demand constraint(s)}
relaxed, resulting in one pricing subproblem for each unit.
% it obtained.
In this paper we focus on the pricing subproblem for thermal units at EDF, a major French
electricity producer.
% \ldots \\
% The objective of the resulting subproblem
Our objective
% , called \red{single} unit commitment problem,
is to determine an optimal two-day production plan that minimizes the overall cost while
respecting several non-linear operational constraints.
% \\
The pricing problem is generally solved by dynamic programming.
\AP{However, due to the curse of dimensionality, dynamic programming reaches its limits when extra-constraints have to be enforced.}
% \comAP{We can take all the constraints}
We model the subproblem as a resource constrained shortest path (RCSP) problem.
Leveraging on RCSP
algorithms recently introduced by the second author,
% \AP{we are able to deal with extra-constraints limiting the gas consumption by groups of units}, and
% by the last author
 we obtain an order of magnitude speed-up \AP{with respect to traditional RCSP algorithms.}

 % \AP{Finally, we show that unit commitment for a group of units sharing a
 % common gas reservoir can be very efficiently using a Dantzig-Wolfe
 % decomposition approach, and that such groups can therefore be considered as a
 % single pricing subproblem in the multi-unit commitment problem.}

\end{abstract}

% \inlAP{a choice must be made between plant and unit, and between plan and planning.}

%% file: intro.tex
%!TEX root=../document.tex

\section{Introduction}

% \comAP{Axel starts to write}
% \todo{don't use the wording of EDF. write easy and short}
% \begin{itemize}
% \item describe EDF global problem and our scope
% \item describe current state and the missing ability to compute solutions including gas
% \item mention time constraints on the overall solving time
% \item literature review
% \item contribution (important)
% \begin{itemize}
% \item faster algorithm for one plant
% \item combinatorial approach
% \end{itemize}
% \item \red{It might be relevant to present the problem as the one of }
% \end{itemize}

\subsection{Context}

A key application of Operation Research to the energy industry is the design of
production plans for electricity producers. A production plan of a unit is
the sequence of production levels at which it operates. Technical constraints
restrict the changes of power levels that can be operated: for instance, a nuclear
unit cannot be launched or stopped instantaneously. The revenue and the costs
generated by a unit depend on its production plan. Given electricity prices, the
\emph{single unit commitment} problem aims at finding a profit maximizing production
plan for a set of units.% while meeting a given electricity demand.
% \todo{PB : you have either prices OR demand. In the multi-unit commitment, you have a
% demand you want to meet at the smallest production cost.
% In the single unit, you have prices and you want to maximize your (profit - operating costs).
% When doing the Lagrangian relaxation of the multi-unit pb (with demand), you obtain a sequence
% of single unit pb (with prices).
% }
% \mk{\emph{Unit commitment} problem aims at building the production plans of all the
% units of a producer in order to meet a given demand at minimum cost.}

As mentioned by \citet{tahanan2014large} in their recent survey on the \uc,
one of the state of the art methods to solve \AP{the multi} unit commitment problem is
to perform a Lagrangian relaxation of the demand constraints, resulting in one
single unit commitment subproblem for each unit. This is for instance the
approach in use at EDF \citep{hechme2010short}.
%a major
%the main
% French electricity producer. 
For producers with a large number of power units, hundreds of single unit commitment
subproblems must be solved at each iteration of the Lagrangian relaxation.
% The actual time is not 1 hour, and this is a sensitive information
%, and the total time needed by the method must not exceed \red{one} hour.
It is
therefore practically crucial to be able to solve single unit commitment
problems very efficiently.
Single unit commitment problems differ between units of different
technologies---hydro, nuclear, thermal, etc.---but are pretty similar between units of the
same technology. The present paper is the {outcome}
%\todo{PB : is "fruit" correct in english ?}
of a partnership with EDF and
focuses on thermal power units.
This industrial context introduces two optimization challenges:
first, the solution must respect several kind of industrial constraints (including ramp-up/ramp-down constraints, minimum online time, minimum offline time, maximal number of startups and maximal number of shutdowns),
and second, the solution scheme computing time must not exceed a few milliseconds.
The goal of the partnership was to design an algorithm that could solve the problem with all the technical constraints in at most 5 milliseconds.

{Dynamic programming is the usual method to solve thermal single unit
commitment problems.}
{Due to the curse of dimensionality}, the usual dynamic programming approach,
% PB : sensitive information
% which is currently in use at EDF,
suffers from several shortcomings,
most notably high computing times when all operational constraints
are taken into account.
In our setting, using dynamic programming requires to drop some technical constraints, as it would otherwise lead to a state space of dimension 7, which is not tractable in practice.
% \AP{For instance, in \Cref{rem:dynamicProgramming}, we explain that taking into account all EDF technical constraints leads to a state space of dimension 7, and dynamic programming leads to very large memory consumption and computing time}.
% Some constraints are therefore dropped in the current scheme, requiring experts to manually update the solution returned by the solver.
% Usual RCSP
% \AP{For instance, constraints on the joint gas consumption by a group of thermal unit sharing a common gas reservoir lead to single unit subproblems that are not well solved by dynamic programming.}
% some thermal units share a common gas reservoir, and
% their respective production plans must therefore jointly satisfy some fuel
% consumption linking constraints.
% including high computation time and
%inability to take into account all relevant constraints.
% Single unit-commitment problem for thermal units is generally solved by dynamic programming.
We therefore introduce a new approach based on resource constrained shortest
paths (RCSP). Standard RCSP algorithms enable to solve the problem with all
technical constraints in a few dozens of milliseconds, which is too slow given
our time limit.
By introducing a new RCSP algorithm, we obtain an order of magnitude speed-up and
are able to solve the problem with all its technical constraints in the desired
time limit.

\subsection{Literature review}  % (fold)
\label{sub:literature_review}
% \begin{itemize}
% 	\item Single unit commitment problem.
% 	\item Resource constrained shortest path approaches to thermal unit commitment problem
% 	\item Resource constrained shortest path problems.
% \end{itemize}

% The \uc have been widely studied. State of the art approaches to the multi-units problem often rely on Lagrangian relaxations, as detailed by the survey of \citet{tahanan2014large}.
% This is notably the case at EDF, as mentioned by \citet{hechme2010short}, who describe the problem and solution method presently in use at EDF.
% % The advantage of the latter approach is that it enables to decouple units, and  Lagrangian decomposition algorithms solve at each step several single-unit subproblems.
% This technique enables to relax the linking constraints between units, and leads to a \suc problem for each unit. As each of these subproblems are solved hundreds of times along the Lagrangian relaxation algorithm, it is crucial to develop efficient algorithms for the \suc.

The single unit commitment problem for thermal units is usually solved using dynamic
programming
\cite{frangioni2006solving,travers1998dynamic,fan2002new,bannister1991rapid}.
The dynamic programming algorithm can be used to solve the whole \suc or only
to take commitment decisions, production levels are chosen later using linear
programming~\cite{frangioni2006solving}. Mixed Integer Linear Programming
approaches have also been proposed as an alternative to dynamic
programming~\cite{morales2013tightStart,morales2013tightBis,ostrowski2012tight,rajan2005minimum,arroyo2000}.
Nonetheless, real-life operational constraints make this problem difficult to
model as a mixed integer linear program, and integrality constraints reduce the
efficiency of this approach.

%%%%%%%%%%%%%%%%%%%%%%%%%%%%%%%%%%%% Fuel Consultion
% Fuel consumption constraints by single units
% \citep{aoki1987unit,aoki1989optimal,tong1989combination} or groups of units
% sharing the same reservoir \citep{cohen1987method,fu2005long} have also been
% considered in the unit commitment literature. When fuel consumption constraints
% are on a single unit, they are generally treated within the subproblem of the
% unit: \citet{aoki1987unit,aoki1989optimal} and \citet{tong1989combination}
% solve this subproblem by Lagrangian relaxation on the fuel consumption
% constraint. If the fuel consumption constraints are related to a group of units, they are
% generally relaxed in the master problem of the full unit-commitment problem
% \citep{cohen1987method,fu2005long}.

\citet{irnich2005shortest} survey the literature on resource constraint
shortest path problems. Enumeration algorithms are the most efficient state of
the art approach to these problems. These algorithms rely on dominance
relations between paths resources to discard partial solutions in an
enumeration of all the paths.
It is well-known that the utilization of bounds on paths resources, used to discard more
paths using a branch and bound paradigm, drastically speed-up enumeration
algorithms \citep{dumitrescu2003improved}.
Enumeration algorithms mainly differ by the order in which partial paths are
processed. Label setting algorithms extend all the paths ending in one vertex
simultaneously, while label correcting algorithms select the ``most promising''
partial path to extend. Label correcting algorithms are more efficient if there
is a good criterion to evaluate the ``most promising'' path, while label
setting algorithms is more efficient otherwise because it never expand paths
that are later dominated. When available, lower bounds on paths resources
provide a natural way to identify the most promising path and make label
correcting algorithms very efficient.
% \comAP{add more references on the different algorithm}

To the best of our knowledge, resource constrained shortest path problems have
not been used to solve the \suc. We believe that this is because, as we show in
Section~\ref{sub:numericcal_results_single} the key element in the performance of our
constrained shortest path approach is the use of bounds to discard partial
paths and to select the ``most promising'' path. Indeed, until recently,
there was no \AP{generic} method
% \todo{PB : a bit early to say that it is a standard method}
% \comAP{in the sense, there were only problem-specific methods, and there was no such ad-hoc method for single unit commitment}
to compute such bounds for non-linear constraints, which
prevented their utilization for the \suc.
We are able to compute such bounds by using the algorithms recently introduced by the
second author \citep{parmentier2014stochastic}.
% to compute bounds for non-linear constrained shortest path problem.

% \inlAP{Add something on linking constraints for gaz units -- literature review on that topic}

% \citep{aoki1989optimal,aoki1987unit}

% When applied to \suc, the use of bounds to discard partial solutions and identify the most promising

% Until recently, there was no standard method to compute such bounds for non-linear constraints, which prevented their use for \suc.

% \inlAP{Explain label setting and label correcting algorithms}

\subsection{Contributions and plan} % (fold)
\label{sub:contributions_and_plan}

Our main contribution is to introduce a new solution method for the single unit
commitment problem which leads to an order of magnitude speed-up with respect
to existing approaches.
This new algorithm is able to solve industrial instances with all their technical constraints within the time limit of 5 milliseconds.
To that purpose,
we have extended the RCSP framework recently introduced by the second author
\citep{parmentier2015algorithms} to be able to deal with constraints on
subpaths and have modeled the single unit commitment problem in this extended
framework. We have then leveraged on the framework to design efficient solution
algorithm.
% As a secondary contribution, we introduce a Dantzig-Wolfe decomposition approach to deal with linking constraints limiting the joint gas consumption by a group of unit sharing the same reservoir. It enables to solve the unit commitment problem for the group of units and with the linking constraints within the time limit of 5 ms.
Section~\ref{sec:single_unit_commitment_problem_statement} states the single
unit commitment problem, and Section~\ref{sec:ordered_monoid_framework} briefly
introduces the lattice ordered monoid framework.
Section~\ref{sec:solution_approach_to_single_unit_commitment} explains how to
model and solve the single unit commitment within the RSCP framework and
Section~\ref{sub:numericcal_results_single} provides numerical results showing
the efficiency of the approach.

%% file: problem.tex
\section{Single unit commitment problem statement}
\label{sec:single_unit_commitment_problem_statement}
%%%%%%%%%%%%%%%%%%%%%%%%%%%%%%%%%%%%%%%%%%%%%%%%%%%%%%%%%%%%%%%%%%%%%%%%%%%%%%%
%%%%%%%%%%%%%%%%%%%%%%%%%%%%%%%%%%%%%%%%%%%%%%%%%%%%%%%%%%%%%%%%%%%%%%%%%%%%%%%

% \Cref{fig:production_plan_modification} illustrates the
% following problem definition.

\subsection{Basic model}

Our goal is to find a production plan of a thermal unit which maximizes the profit obtained
by selling electricity while respecting several technical constraints.

We consider a finite time horizon $[T] = \{1, 2, \dots, T \}$ 
with $T \in \mathbb{N}_{\ge 1}$.
The unit can produce at power \emph{levels} $s$ to be chosen in a finite set $S$.
%The plant is allowed to switch between power levels only throught one of its allowed
%\emph{transitions} $a \in \calA$
%
The unit can switch between power levels by choosing a \emph{transition}
$\alpha = (s^{\mathrm{init}}_{\alpha}, s^{\mathrm{final}}_{\alpha},\tau_{\alpha} )$ among
the finite set $\calA \subset S \times S \times [T]$ of allowed transitions.
A transition $\alpha \in \calA$ is defined by an \emph{initial level} $s^{\mathrm{init}}_{\alpha} \in S$,
a \emph{final level} $s^{\mathrm{final}}_{\alpha} \in S$, and a \emph{duration} $\tau_{\alpha} \in [T] $.

% At the first date $t^{\mathrm{b}} = 1$, the plant is at one of its power level $s_1 = s^{\mathrm{b}} \in S$. It then must chose an allowed transition $\alpha_0 \in \calA$ such that $s^{\mathrm{init}}_{\alpha_0} = s_0$. The plant will then follow this transition and attain the power level $s_1 = s^{\mathrm{final}}_{\alpha_0}$ at time $t_1 = t_0 + \tau_{\alpha_0}$.
% It will then chose another allowed transition with $s_1$ as initial state, and reach a level $s_2$ at time $t_2$, and so on until the end of the time span is reached.
At the beginning of the time horizon $t^{\mathrm{b}} = t_1 = 1$, the unit is at
power level $s^{\mathrm{b}} = s_1 \in S$. It then must chose an
allowed transition $\alpha_1 \in \calA$ such that $s^{\mathrm{init}}_{\alpha_1}
= s_1$. The unit will follow this transition and attain the power level
$s^{\mathrm{final}}_{\alpha_1} = s_2$ at time $t_2 = t_1 + \tau_{\alpha_1}$.
After this it will chose another allowed transition with $s_2$ as initial
level, reach power level $s_3$ at time $t_3$, and so on until the end of the
time horizon is reached.

\begin{definition}\label{def:prod_plan}
A \emph{production plan} $p$ is a sequence of the form
\begin{equation*}%\label{eq:prod_plan_def}
 (s_1,t_1, \alpha_1), (s_2, t_2, \alpha_2), \ldots, (s_k, t_k, \alpha_k)
\end{equation*}
where $k \in \N_{\ge 1}$, $(s_i,t_i, \alpha_i) \in S \times [T] \times \calA$ for all $i \in
\{1,\ldots,k\}$ and the six following conditions are satisfied
\begin{alignat*} {3}
  &s^{\mathrm{init}}_{\alpha_i} = s_i &\quad& \forall i \in \{ 1, \dots, k\} ,
 &\quad\quad\quad\quad\quad\quad\quad\quad\quad\quad
 &s_1 = s^{\mathrm{b}}, 
 \\
 &s^{\mathrm{final}}_{\alpha_i} = s_{i+1} && \forall i \in \{1, \dots, k-1\},
 && 
 t_1 = 1, 
 \\
 & t_i + \tau_{\alpha_i} = t_{i+1} && \forall i \in \{ 1, \dots, k-1\},
 && t_k + \tau_{\alpha_k} = T.
\end{alignat*}

% \begin{tabular}{p{0.48\textwidth}p{0.48\textwidth}}
% \begin{itemize}
%  \item $s_1 = s^{\mathrm{b}}$,
%  \item $t_k + \tau_{}$
%  \item $(s_i,t_i, \alpha_i) \in S \times [T] \times \calA$ for all $i = 1, \dots, k$
%  \item 
%  $s^{\mathrm{init}}_{\alpha_i} = s_i$ for all $i = 1, \dots, k$,
%  \item $s^{\mathrm{final}}_{\alpha_i} = s_{i+1}$ for all $i = 1, \dots, k-1$
%  % \item 
%  $t_i + \tau_{\alpha_i} = t_{i+1}$ for all $i = 1, \dots, k-1$.
% \end{itemize}
% &
% \begin{itemize}
%  \item $s_1 = s^{\mathrm{b}}$,
%  \item $(s_i,t_i, \alpha_i) \in S \times [T] \times \calA$ and %for all $i = 1, \dots, k$
%  % \item 
%  $s^{\mathrm{init}}_{\alpha_i} = s_i$ for all $i = 1, \dots, k$,
%  \item $s^{\mathrm{final}}_{\alpha_i} = s_{i+1}$, for all $i = 1, \dots, k-1$
%  % \item 
%  $t_i + \tau_{\alpha_i} = t_{i+1}$ for all $i = 1, \dots, k-1$.
% \end{itemize}
% \end{tabular}

\end{definition}
% \todo{PB : should we precise that $t_0 = 1$ and $t_{final} > T$ ?  Or should we say nothing if "partial production plans" are needed ?}
% \begin{rem}
Note that the set of allowed transitions is restricted to model several technical
constraints of the thermal unit.
For example, starting up a thermal unit requires several time steps to warm up the physical components. Consequently, the  allowed transitions that have the offline level as their initial level have a duration of several time steps.
Another example comes from the so-called ramp-up and ramp-down constraints
\citep{rajan2005minimum}, which prevent the unit from increasing or decreasing
its power output too quickly.
% Consequently, the allowed transitions that have a "low" power level as their initial level and "high" one as their final level requires several time steps.
Consequently, the allowed transitions with a large difference in power output
between the initial and final power level require several time steps.
Furthermore, the duration of a transition can be used to model the minimum time
that need to be spent at the same power level after a change. 
% We can also model minimum time at power level through transition durations.
%
The set of transitions contains not only the switches from one power level to another, but also the transitions that allows to stay on the same power level. More precisely, all transitions of the form $(s, s, 1)$ for $s \in S$ are allowed (and thus contained in $\calA$).
% \end{rem}

\subsection{Minimum up/down times}

The set of power levels $S$ is partitioned into disjoint subsets called
\emph{layers}. Moreover, the set of layers is partitioned into disjoint subsets
 called \emph{modes}.
A mode should be understood as ``the unit is online'' or ``the unit is
offline''.
The rationale behind the layers is as follows: all power levels
in a layer correspond to the same global power output for the unit, but the global
power output can be distributed between active power and system services.
Consequently, a change of power level within a layer does not have the same technical
impact as a change between two distinct layers.

We denote by $L$ the set of layers and by $\mathsf{layer}(s) \in L$ the layer associated
to level $s \in S$.
Similarly, $M$ is the set of modes and $\mathsf{mode}(s) \in M$ is the mode associated
to $\mathsf{layer}(s)$.
When reaching a new layer, the unit has to stay within this layer
for at least the duration $\tau_{\mathrm{lay}} \in \N$. Similarly, when reaching a new
mode,
the unit has to stay within this mode for at least $\tau_{\mathrm{mod}} \in \N$ time steps.
More formally, a production plan 
% of the form \ref{eq:prod_plan_def} 
must satisfy the following constraints.

\begin{enumerate}[label=(\Alph*)]
\item \label{rule:min_layer_duration} For any $i \in \{1, \dots, k-1\}$ such that
$\alpha_i$ is a change of layer
(i.e.,
 $\mathsf{layer}(s^{\mathrm{init}}_{\alpha_i}) \neq \mathsf{layer}(s^{\mathrm{final}}_{\alpha_i})$
 ),
 let $\alpha_j$ be the next change of layer (i.e., the smallest $j \in \{i+1,
   \ldots, k \}$ such that $\mathsf{layer}(s^{\mathrm{final}}_{\alpha_j}) \neq
 \mathsf{layer}(s^{\mathrm{final}}_{\alpha_i})$). Then, we must have $t_j  \geq t_{i+1} + \tau_{\mathrm{lay}}$
 % let $l = \mathsf{layer}(s^{\mathrm{final}}_{\alpha_i})$ be the target layer, and  $\alpha_j$ be the next change of layer
 % (i.e. the smallest $j \in \{i+1, \dots, k \}$ such that
 % $\mathsf{layer}(s^{\mathrm{final}}_{\alpha_j}) \neq l$).

 \item \label{rule:min_mode_duration} For any $i \in \{1, \dots, k-1\}$ such that
 $\alpha_i$ is a change of mode
 (i.e.,
  $\mathsf{mode}(s^{\mathrm{init}}_{\alpha_i}) \neq \mathsf{mode}(s^{\mathrm{final}}_{\alpha_i})$
  ),
  % let $m = \mathsf{mode}(s^{\mathrm{final}}_{\alpha_i})$ be the target mode, and let $\alpha_j$ be the next change of mode
  % (i.e. the smallest $j \in \{i+1, \dots, k \}$ such that
  % $\mathsf{mode}(s^{\mathrm{final}}_{\alpha_j} \neq m$).

  let $\alpha_j$ be the next change of mode
  (i.e., the smallest $j \in \{i+1, \dots, k \}$ such that
  $\mathsf{mode}(s^{\mathrm{final}}_{\alpha_j} \neq
    \mathsf{mode}(s^{\mathrm{final}}_{\alpha_i})$.
  Then, we must have $t_j  \geq t_{i+1} + \tau_{\mathrm{mod}}$
\end{enumerate}

% \begin{rem}
The minimum mode duration $\tau_{\mathrm{mod}}$ allows to model the minimum up/down times found in the
literature \citep{rajan2005minimum}. % \todo{ref needed, e.g. "Minimum up/down polytopes of the unit commitment problem with start-up costs"
% by Rajan and Takriti}
% \end{rem}
% \begin{rem}
% \todo{PB: Not sure if min layer time constraints exists in the literature}
% \end{rem}

\subsection{Maximum number of layer/mode changes}
Transitions from one layer to another, or from one mode to another (e.g. a startup),
or from a "high" power level to a "low" one (so-called \emph{deep} transitions) create strain
on the physical components of a plant.
 To limit this strain, we limit the
number of such changes that can happen within the time horizon.
We consider three constraints of this kind:
% \begin{itemize}
  the maximum number of startups,
  the maximum number of layer changes,
  and the maximum number of deep transitions.
% \end{itemize}
More formally, for each transition
$\alpha \in A$, we are given
% an associated  number of startups $n_{\mathsf{startup}}(\alpha) \in \{0,1\}$,
% a number of change of layer $n_{\mathsf{layer}}(\alpha) \in \{0,1\}$,
% and a number of deep transitions $n_{\mathsf{deep}}(\alpha) \in \{0,1\}$.
an associated indicator function for startups
$\ind_{\mathsf{startup}}(\alpha)$, layer changes
$\ind_{\mathsf{layer}}(\alpha)$, and deep transitions
$\ind_{\mathsf{deep}}(\alpha)$.
We also have as input
a maximum number of startups $n^{\mathrm{max}}_{\mathsf{startup}} \in \N$,
a maximum number of layer changes $n^{\mathrm{max}}_{\mathsf{layer}} \in \N$,
and a maximum number of deep transitions $n^{\mathrm{max}}_{\mathsf{deep}} \in \N$.
A production plan must satisfy the following constraints.
\begin{enumerate}[label=(\Alph*)]
\setcounter{enumi}{2}
\item \label{rule:max_nb_of_startups} $\displaystyle\sum_{i = 1}^k \ind_{\mathsf{startup}}(\alpha_i) \leq n^{\mathrm{max}}_{\mathsf{startup}}$

\item \label{rule:max_nb_of_change_of_layer} $\displaystyle\sum_{i = 1}^k \ind_{\mathsf{layer}}(\alpha_i) \leq n^{\mathrm{max}}_{\mathsf{layer}}$

\item \label{rule:max_nb_of_deep_transitions} $\displaystyle\sum_{i = 1}^k \ind_{\mathsf{deep}}(\alpha_i) \leq n^{\mathrm{max}}_{\mathsf{deep}}$

\end{enumerate}
\noindent
A production plan is \emph{feasible} if it satisfies constraints
\ref{rule:min_layer_duration}--\ref{rule:max_nb_of_deep_transitions}.

% \begin{rem}
% We also model minimum time at power level through transition durations.
% \end{rem}

\subsection{Costs and profits}
Running a unit induces several kind of production costs, including: startup costs,
fixed cost for being online, and a cost proportional to the power output.
At the same time, the produced energy generates profit, which depends
not only on the amount of produced energy, but also on the time steps at which the
energy is produced and sold.
Taking into account all these costs and gains result in a global cost function
$\lambda: [T] \times \calA \rightarrow \R$, where $\lambda(t, \alpha)$ is the cost induced
by choosing transition $\alpha \in \calA$ at time step $t \in [T]$.
More formally, the overall cost of a production plan 
% of the form \ref{eq:prod_plan_def} 
is given by

\begin{equation}\label{eq:cost_def}
  \displaystyle\sum_{i = 1}^k \lambda(t_i, \alpha_i).
\end{equation}

%%%%%%%%%%%%%%%%%%%%%%%%%%%%%%%%%%%%%%%%%%%%%%%%%%%%%%%%%%%%%%%%%%%%%%%%%%%%%%%
\subsection{Single unit commitment problem}
\label{sub:single_unit_commitment_problem}
%%%%%%%%%%%%%%%%%%%%%%%%%%%%%%%%%%%%%%%%%%%%%%%%%%%%%%%%%%%%%%%%%%%%%%%%%%%%%%%

The \emph{single unit commitment problem} consists in finding a feasible production plan 
% satisfying constraints \ref{rule:min_layer_duration}--\ref{rule:max_nb_of_deep_transitions} 
of minimum cost \eqref{eq:cost_def}.
 % - decision dates T
 % - at the first date, the plant is at one of its power level (chosen in a discrete set)
 % - it must then chose a transition in the set A
 % - each transition has an initial state, a final state, and a duration
 % - a transition can be chose only if the plant is in the initial state
 % - the plant then follows the transition and end up in the final state at the date start_date + duration
 %
 % - a production plan is thus a sequence of transitions
 %
 % - levels are partitioned in "layers", and layers are partioned in "modes"
 % - minimum time in each level, in each layer, and in each mode
 %
 % - max number of "change of layer" and max number of "change of mode"
 %
 % - a map (transition, date) -> gain
 %
 % - at each decision date, the plant must be at one of its power level (chosen in a dicrete set)
 % - between two time steps, is transition. A transition has
 % - unit can produce at different power levels in a discrete set
 % - at each t

%% file: monoid.tex
%!TEX root=../document.tex

\section{Ordered Monoid Framework} % (fold)
\label{sec:ordered_monoid_framework}

% \section{Resource Constrained Shortest Path algorithm} % (fold)
% \label{sec:resource_constrained_shortest_path_algorithm}
In this section, we follow \citep{parmentier2017aircraft} to briefly introduce the resource constrained shortest path framework of \cite{parmentier2015algorithms}.
% Section \ref{sub:a_useful_monoid} extends it to deal with constraint on subpaths, constraints that are needed to model the single unit commitment problem in Section \ref{sec:single_unit}.
% \subsection{Resource constrained shortest path framework} % (fold)
% \label{sub:resource_constrained_shortest_path_framework}
% subsection resource_constrained_shortest_path_framework (end)
% \subsection{Framework and algorithm} % (fold)
% \label{sub:framework_and_algorithm}
% \todo{New version. Needs to be polished}
% subsection framework_and_algorithm (end)
{A \emph{digraph} $D$ is a pair $(V,A)$, where $V$ is the set of \emph{vertices} and $A$ is the set of \emph{arcs} of $D$. An \emph{arc} $a$ links a \emph{tail} vertex to a \emph{head} vertex.
% An arc $a$ is \emph{incoming} to (resp. outgoing from) $v$ if $v$ is the \emph{head} (resp. the tail) of $v$.
% The set of arcs incoming to (resp. outgoing from) $v$ is denoted $\delta^{-}(v)$ (resp. $\delta^{+}(v)$).
A \emph{path} is a sequence of arcs $a_{1},\ldots,a_{k}$ such that for each $i \in \{1,\ldots,k-1\}$, the head vertex of $a_{i}$ is the  tail vertex of $a_{i+1}$.
% Note that with this definition, paths can contain multiple copies of an arc or of a vertex. A path $P$ is said to be \emph{elementary} if it contains at most one copy of each vertex.
The \emph{origin} of a path it the tail of its first arc and its
\emph{destination} is the head of its last arc. Given two vertices $o, d \in
V$, an $o$-$d$ path $P$ is a path with origin $o$ and destination $d$.}

A binary operation $\rplus$ on a set $\rset$ is \emph{associative} if $\re
\rplus (\re' \rplus \re'') = (\re \rplus \re')\rplus \re''$ for $\re,\re',\re''
\in \rset$. An element $0$ is \emph{neutral} if $0 \rplus \re =  \re
\rplus 0 = \re$ for any $\re \in \rset$. A set $(\rset,\rplus)$ is a
\emph{monoid} if $\rplus$ is associative and admits a neutral element. A partial
order $\rleq$ is \emph{compatible} with $\rplus$ if the mappings $\re \mapsto
\re \rplus \re'$ and $\re \mapsto \re' \rplus \re$ are non-decreasing for all
$\re' $ in $\rset$. A partially ordered set $(\rset,\rleq)$ is a \emph{lattice}
if any pair $(q,q')$ of elements of $\rset$ admits a greatest lower bound or
\emph{meet} denoted $\re\meet\re'$, and a least upper bound or \emph{join}
denoted $\re\join\re'$. A set $(\rset,\rplus,\rleq)$ is a \emph{lattice ordered
monoid} if $(\rset,\rplus)$ is a monoid, $(\rset,\rleq)$ is a lattice, and
$\rleq$ is compatible with $\rplus$.

Given a lattice ordered monoid $(\rset,\rplus,\rleq)$, a digraph $D=(V,A)$, an
origin vertex $o$, a destination vertex $d$, a set of arc resources $\re_a \in \rset$ for all arcs $a \in A$ and two non-decreasing mappings
$\rcost : \rset \rightarrow \R$ and $\rmeas:\rset \rightarrow \{0,1\}$, the
\MRCSP seeks
%\todo{PB : "path" needs to be defined}
$$\text{an $o$-$d$ path $P$ of minimum
$\rcost\left(\bigrplus_{a\in P}\re_{a}\right)$ among those satisfying $\rmeas\left(\bigrplus_{a\in
P}\re_{a}\right) = 0$,} $$ 
where the sum $\bigrplus_{a\in P}$ is taken along path $P$ -- operator $\rplus$ is not necessarily commutative.
The sum $\bigrplus_{a\in P}\re_{a}$ is the
\emph{resource} of a path $P$, and we denote it by $\re_{P}$. The real number
$\rcost\left(\re_P\right)$ is its \emph{cost}, and the path $P$ is
\emph{feasible} if $\rmeas\left(\bigrplus_{a\in P}\re_{a}\right)$ is equal to
$0$. We therefore call $\rcost$ and $\rmeas$ the cost and the infeasibility
functions.
% When the last vertex of $P$ is the first vertex of $Q$, we denote by
% $P+Q$ the path $P$ followed by the path $Q$\comAP{check if used. }, and $\re_{P}$ the resource $\bigrplus_{a\in
% P}\re_{a}$ of a path $P$.

We now describe an \emph{enumeration algorithm} for the \MRCSP. It follows the
standard labeling scheme \cite{irnich2005shortest} for resource constrained
shortest paths. The lattice ordered monoid structure enables to extend these
algorithms to new problems and to speed them up due to new tests and keys. A
list $\mathsf{L}$ of partial paths $P$ and an upper bound $c_{od}^{\mathrm{UB}}$ on the
cost of an optimal solution are maintained. Initially, $\mathsf{L}$ contains the
empty path at the origin $o$, and $c_{od}^{\mathrm{UB}} = +\infty$. The
\textsf{key} in Step \ref{step:key} and the \textsf{test(s)} in Step
\ref{step:test} must be specified to obtain a practical algorithm and will be
defined below. We now
introduce two \textsf{tests}.
 While $\mathsf{L}$ is
not empty, the following operations are repeated.

\begin{enumerate}[label=(\roman*)]
	\item Extract a path $P$ of minimum \textsf{key} from $\mathsf{L}$. Let $v$ be the last vertex of $P$. \label{step:key}
	\item If $v = d$, then: if $\rmeas(\re_{P}) = 0$ and $\rcost(\re_{P}) <
    c_{od}^{\mathrm{UB}}$, update $c_{od}^{\mathrm{UB}}$ to $\rcost(\re_{P})$. \label{step:destination}
	\item Else if \textsf{test(s)} return(s) \textsf{yes}, extend $P$: for each arc $a$ outgoing from $v$, add $P+a$ to $\mathsf{L}$. \label{step:test}
  % \todo{PB: why the "(s)" ?} \todo{PB: why the "(s)" ?}
\end{enumerate}
A path $P$ \emph{dominates} a path $Q$ if $\re_{P}\rleq \re_{Q}$. The
\emph{dominance \textsf{test}} uses a list $\mathsf{L}_{v}^{\mathrm{nd}}$ of
non-dominated $o$-$v$ paths for each vertex $v$. It can be expressed as follows.

\smallskip (Dom) \emph{If there is no path in $\mathsf{L}_{v}^{\mathrm{nd}}$ that dominates $P$, return} \textsf{yes}. \emph{Otherwise return} \textsf{no}.
% \emph{Is there no path in $\mathsf{L}_{v}^{\mathrm{nd}}$ that dominates $P$?} %\inlAP{If... otherwise...}
% \emph{Is $P$ dominated by no path in
% $\mathsf{L}_{v}^{\mathrm{nd}}$?}
% \todo{PB : awkward sentence. Maybe : " There is no path in $\mathsf{L}_{v}^{\mathrm{nd}}$ that dominates $P$"}
\smallskip

% \begin{equation} \text{Is $P$ non-dominated by any path in
% $\mathsf{L}_{v}^{\mathrm{nd}}$?} \label{eq:Dom}\tag{Dom} \end{equation}

% \begin{itemize} \item[(Dom)] Is $P$ non-dominated by any path in
% $\mathsf{L}_{v}^{\mathrm{nd}}$? \label{eq:Dom} \end{itemize}

% I want to speak of \ref{eq:Dom}

\noindent If the answer is \textsf{yes}, then before extending $P$, we remove all the
paths dominated by $P$ from $\mathsf{L}_{v}^{\mathrm{nd}}$ and add $P$ to
$\mathsf{L}_{v}^{\mathrm{nd}}$.
 We now suppose to have lower bounds $b_v$ on
the resource $\re_Q$ of all the $v$-$d$ paths $Q$ for each $v$ in $V$, and introduce the following \emph{lower bound \textsf{test}}.%  in Step~\ref{step:test}.

\smallskip (Low) \emph{If $P$ satisfies
$\rmeas(\re_{P}\rplus b_v) =0$ and $\rcost(\re_{P}\rplus b) \leq
c_{od}^{\mathrm{UB}}$ return} \textsf{yes}. \emph{Otherwise return} \textsf{no}.
% \emph{Does $P$ satisfy
% $\rmeas(\re_{P}\rplus b_v) =0$ and $\rcost(\re_{P}\rplus b) \leq c_{od}^{UB}$?}
\smallskip

% If the answer to this test is \textsf{no}, then any $o$-$d$ path starting by
% $P$ is either infeasible or of cost greater than $c_{od}^{UB}$.
% We have the following result.

% As noted in Section 8 of \citep{parmentier2015algorithms}, we have the following result.

\begin{prop}\label{prop:MRCSPalgoConvergence}
 % Independently of the key and of the test(s) used among those mentioned,  if $D$ is acyclic,
Suppose that $D$ is acyclic. Then
if none, one, or both of the tests  \nf{(Low)} and \nf{(Dom)} are used,
% for any combination of the tests \nf{(Low)}
%   and \nf{(Dom)} \todo{PB : what is a combination of these two things ?},
  the algorithm converges after a finite number of iterations,
  and, at the end of the algorithm, $c_{od}^{\mathrm{UB}}$ is equal to the cost of an
  optimal solution of the \MRCSP if such a solution exists, and to $+\infty$
  otherwise.
\end{prop}
% \todo[inline]{PB : a proposition either have a proof (at least a sketch of), or cite a reference with a proof}
We underline that the result does not depend on the key used. The proof being rather technical, we only sketch it here to underline the main ideas. Details are available in the preamble of Section 8 in \citep{parmentier2015algorithms}. Practical choices of key and tests are given after the proof.

\begin{proof}[Sketch of the proof]
The enumeration scheme ensures that an $o$-$v$ path is considered at most once. The finite number of paths in an acyclic graph ensures convergence.
The update mechanism of $c_{od}^{\mathrm{UB}}$ ensures that, at the end of the
algorithm, $c_{od}^{\mathrm{UB}}$ is equal to the cost of the best $o$-$d$ path considered.
A given path $P$ is considered if and only if none of its subpaths is discarded by the tests. Proving that there exists an optimal path $P$ whose subpaths are not discarded is rather technical, but relies on simple ideas.
The dominance test relies on the fact that there is an
optimal solution whose subpaths are all non-dominated. Such a path can be built using an easy recursion. The lower bound test comes from that fact that $\re_P \rplus b_v$ is a lower bound on the resource of any $o$-$d$ path starting by $P$. Hence, if $P$ is the subpath of an optimal path, it satisfies
$\rmeas(\re_{P}\rplus b_v) =0$, and $\rcost(\re_{P}\rplus b) \leq \rcost(Q)$
for any feasible $o$-$d$ path $Q$, giving $\rcost(\re_{P}\rplus b) \leq
c_{od}^{\mathrm{UB}}$.
 % We therefore only need to prove that there exits an optimal path $P$ that is considered.
\end{proof}

\citet{irnich2005shortest} indicate that a traditional choice in the literature is to use the \textsf{key} $\rcost(\re_P)$ for Step~\ref{step:key} the dominance \textsf{test} \nf{(Dom)} in Step~\ref{step:test}. The lattice ordered monoid framework is not required to use this test and this key.
%  the key
% $\rcost(\re_P)$ and the dominance test.
The main advantage of this framework is
that it enables to use the practically efficient algorithm in
\cite{parmentier2015algorithms} to compute good quality lower bounds $b_v$ on
the resource $\re_Q$ of all the $v$-$d$ paths $Q$ for each vertex $v$ in $V$.
When these bounds are computed, we can use the lower bound \textsf{test} \nf{(Low)}, and the \textsf{key} $\rcost(\re_P \rplus b_v)$.
This \textsf{key} is a
lower bound on the cost of any $o$-$d$ path starting by $P$.
As we will see in \Cref{sub:numericcal_results_single},  on our problem, the introduction of the lower bound test and of this new key both lead to a drastic speed-up with respect to the traditional choice mentioned above.
This speed-up is not specific to our problem
\citep{parmentier2015algorithms,parmentier2017aircraft}, and probably comes
from the fact that \nf{(Low)} discards many paths and $\rcost(\re_P \oplus b_v)$ is a good evaluation of how promising a path $P$ is.
Furthermore, on our problem, computing the bounds takes around $8\%$
$(0.15\text{ms})$ percent of the total solving time of the fastest enumeration
algorithm.
% , can be used
% instead of $\rcost(\re_P)$ in Step~\ref{step:key}.
% Key $\rcost(\re_P +b_v)$ is a lower bound on the cost of any $o$-$d$ path starting by $P$.
% Experiments have shown that, when bounds $b_v$ are well-chosen, for instance
% when they are computed using the algorithm in \citep{parmentier2015algorithms},
% This performance comes from the fact that
% $\rcost(\re_P \oplus b_v)$ is a
% far \todo{PB : how far ?}
% better evaluation of how promising a path $P$ is
% than $\rcost(\re_P)$, and leads to
% much  \todo{ PB : how much ?}
% better result in practice, as we will
% see in the numerical results.

\begin{rem}
The algorithm to compute the lower bounds \citep{parmentier2015algorithms} is practically efficient because, if it requires $O(|V||A|)$ operations of the monoid in the worst case, it only performs $O(|A|)$ operations of the monoid on all examples considered. By ``good quality lower bound'' $b_v$, we mean that $b_v$ is not far from the greatest possible lower bound, which is the meet of the resource of all the $v$-$d$ paths. In fact, if $(\rset,\rplus,\rleq)$ is a dioid \citep{gondran2008graphs}, that is, if $\rplus$ distributes with respect to $\meet$, the bound $b_v$ returned by the algorithm is the greatest lower bound. 
Unfortunately, the monoid we use in this paper is not a dioid.
%  a condition that is satisfied by all the monoids considered in this paper, the bound $b_v$ returned by the algorithm is the greatest lower bound.
% \todo{PB : what do you mean exactly by dioid, and is it the case the our monoid considered here is a dioid ?
\end{rem}

\begin{rem}\label{rem:LabelSetting}
The enumeration algorithm described is also called a label-correcting algorithm
\citep{irnich2005shortest}. A more frequent alternative to solve the resource
constraint shortest path problem is to use a
label-setting algorithm. For instance, the Boost \texttt{C++} library \citep{boost2012libraries}
implements this kind of algorithm. Our enumeration algorithm can be turned
into a label setting algorithm by adjusting the following parts. In Step~\ref{step:key}, instead
of selecting a path, we select a \emph{vertex} $v$ with minimum ``key'' and
extend all the non-dominated paths in $\mathsf{L}_{v}^{\mathrm{nd}}$ by all
  outgoing edges of $v$. Under the assumption that the
  graph is acyclic, which is the case in this paper, vertices
  can be extended according to a topological order on the graph.
   % and each vertex is  touched only once.
\end{rem}

%% file: singleunit.tex
%!TEX root=../document.tex

%%%%%%%%%%%%%%%%%%%%%%%%%%%%%%%%%%%%%%%%%%%%%%%%%%%%%%%%%%%%%%%%%%%%%%%%%%%%%%%
%%%%%%%%%%%%%%%%%%%%%%%%%%%%%%%%%%%%%%%%%%%%%%%%%%%%%%%%%%%%%%%%%%%%%%%%%%%%%%%
\section{Solution approach to single unit commitment}
\label{sec:solution_approach_to_single_unit_commitment}
%%%%%%%%%%%%%%%%%%%%%%%%%%%%%%%%%%%%%%%%%%%%%%%%%%%%%%%%%%%%%%%%%%%%%%%%%%%%%%%
%%%%%%%%%%%%%%%%%%%%%%%%%%%%%%%%%%%%%%%%%%%%%%%%%%%%%%%%%%%%%%%%%%%%%%%%%%%%%%%

\newcommand{\Stay}{\mathsf{Stay}}
\newcommand{\Change}{\mathsf{Change}}
\newcommand{\infeas}{\mathsf{infeasible}}

\newcommand{\stay}{\mathsf{st}}
\newcommand{\change}{\mathsf{ch}}

We now model the single thermal unit commitment problem as a \MRCSP. 

%%%%%%%%%%%%%%%%%%%%%%%%%%%%%%%%%%%%%%%%%%%%%%%%%%%%%%%%%%%%%%%%%%%%%%%%%%%%%%%
\subsection{Digraph}
\label{sub:digraph}
%%%%%%%%%%%%%%%%%%%%%%%%%%%%%%%%%%%%%%%%%%%%%%%%%%%%%%%%%%%%%%%%%%%%%%%%%%%%%%%
Let $D =
 (V,A)$ be the digraph defined as follows. The vertex set $V$ is $S\times [T]
 \cup\{o,d\}$, where $o$ is an origin vertex, and $d$ a destination vertex.
The arc set $A$ is the set of pairs $\bp{(s,t),(s',t')}$ such that $(s,s',t-t') \in \cal A$
completed with an initial arc $\big(o,(s^{\mathrm{b}},1)\big)$, where $s^{\mathrm{b}}$ is
the initial level of the unit, and the final arcs $\big((s,T),d\big)$ for each level $s$ in $S$.

\begin{prop}
There is a bijection between the set of production plans  and the set of $o$-$d$ paths in the digraph $D$.
\end{prop}
\begin{proof}
Consider a production plan as defined in \ref{def:prod_plan}. By definition,
$(s_i,s_{i+1},t_i-t_{i+1}) \in \cal A$ for all $i \in \{ 1, \dots, k-1 \}$. Hence, the path formed with the sequence of arcs
\begin{equation}\label{eq:path_to_plan}
\bp{o,(s_1,t_1)} \bp{(s_1,t_1), (s_2, t_2)} \dots \bp{(s_{k-1},t_{k-1}), (s_k, t_k)} \bp{(s_k, t_k), d}
\end{equation}
exists in the digraph $D$.

Conversely, by definition of the digraph $D$, any $o$-$d$ path is a sequence of arcs of the form \ref{eq:path_to_plan}, to which we can associate a production plan that satisfies \ref{def:prod_plan}.
\end{proof}

\subsection{A monoid for minimum layer duration constraints}

Before turning to the whole thermal unit commitment problem, this section introduces a model to take only into account the constraints on the minimum time on layers \ref{rule:min_layer_duration}.

In order to satisfy the minimum duration in a layer, a counter for the time
spent in a given layer is used. 
When a layer change occurs, this counter is used to check if the time spent in the layer exceeds the minimum duration in a layer $\tau_{\mathrm{lay}}$.
 % remaining before being allowed to perform afr change of layer. When arriving at a new layer, the counter is set to the minimum duration in this layer. For every transition that stays within the layer, the counter is decremented by the duration of the transition. When a change of layer occurs, one can check if the counter is below zero in order to verify that the minimum time in the layer have been fulfilled.
This approach can be modelled with the following monoid.
%domain definition and its interpretation
The underlying set $\sset^{\mathrm{lay}}$ of our monoid is the disjoint union $\Stay \cup \Change \cup \{ \infeas \}$
where $\Stay$ is identified with $\R_+$ and $\Change$ is identified with $\R_+^2$.
An element of $\Stay$ is denoted by $\stay(a)$ with $a \in \R_+$, and an element of $\Change$
by $\change(x, y)$ with $(x,y) \in \R_+^2$.

The semantic of our monoid is as follows: when an arc (resp.~path) of our graph is decorated with an element  $\stay(x) \in \Stay$,
it means that by following this arc (resp.~path), the unit stays in the same layer for a duration $x$.
% \todo{PB : the (resp. path) is a bit misleading. If a path is decorated with $\stay(a)$, it means that it has stayed in its last layer for
% a duration of $a$, but it may have been through several layers before that. }
%, and the duration remaining before being allowed to perform a change of layer is decreased by $x$.
An element $\change(x, y) \in \Change$ decorating an arc means that a change of layer happens on this arc.
When this arc  (resp.~path) is followed, the unit stays in its initial layer
for a duration of $x$, changes the layer, and then stays in the final layer for a duration $y$.
An element $\change(x, y) \in \Change$ decorating a path means that at least one change of layer happens along this path, and that the unit has stayed 
for a duration $x$ in the first layer encountered, and a duration $y$ in the last layer encountered. 
% \todo{PB : same as above, the (resp. path) is a bit misleading. If a path is decorated with $\change(x,y)$, it means that it has stayed 
% in the first layer encountered for a duration $x$, and stayed in the last encoutered layer for a duration $y$. It may have done several
% changes of layers in-between}
% Before the change happens, the plant stays in its current layer for a duration of $x$. After the change, the plant will have to stay in the new layer for a duration $y$.
The element $\infeas$ denotes a violation of a minimum duration in a layer along the arc (resp.~path).

%arc decoration
Following this semantic, the arcs of our graphs are decorated with elements of our monoid as follows:
\begin{itemize}
\item An arc corresponding to a transition $\alpha$ which does not induces a change of layer is decorated with the element $\stay(\tau_{\alpha})$, where $\tau_{\alpha}$ is the duration of the transition $\alpha$.
\item An arc corresponding to a transition which induces a change of layer is decorated with the element $\change(0, 0)$.%, where $l$ is the final layer of the transition.%, and $\tau_{\mathrm{lay}}$ the minimum duration to be respected in layer $l$.
\item The initial arc (from the origin vertex $o$ to the vertex of the initial
  level) is decorated with $\change(0,
  \tau_{\mathrm{lay}}-\tau_{\textrm{init}})$, where $\tau_{\textrm{init}}$ is
  the duration after which the unit is allowed to leave the layer of its
  initial level.
\item Final arcs (i.e. arcs arriving at the destination node) are decorated with $\stay(0)$, the neutral element of our monoid (see below).
\end{itemize}

% Operations

We define our monoid operation $\splus$ as follows :
\begin{subequations}\label{eq:monoidSum}
\begin{align}
&\stay(a) \splus  \stay(b)  = \stay(a + b)         &\text{ for all } a, b \in \R_+ \\
&\stay(a) \splus  \change(x, y) = \change(a + x, y) &\text{ for all } a, x, y \in \R_+\\
&\change(x,y) \splus  \stay(a) = \change(x, y + a) &\text{ for all } a, x, y \in \R_+\\{}
&\change(x, y) \splus  \change(u, v) =  \left\{\begin{array}{ll}
\infeas & \text{if } y + u < \tau_{\mathrm{lay}}, \\
\change(x,v) & \text{otherwise,}
\end{array}\right.  &\text{ for all }  x, y, u, v \in \R_+ \label{eq:infeasibleSum}\\
&\infeas \splus  \stay(a) = \stay(a) \splus  \infeas = \infeas &\text{ for all } a \in \R_+ \\
&\change(x,y) \splus  \infeas = \infeas \splus  \change(x,y) = \infeas &\text{ for all }  x, y \in \R_+
\end{align} 
\end{subequations}

It is worthwhile to intrepret this definition under the light of the semantic of our monoid.
Equation~\eqref{eq:infeasibleSum} means that performing two consecutive changes
of layers can yield an infeasible path if the minimum duration in the layer
in-between is not respected. Otherwise, the minimum duration constraint is
respected for this layer, and we can ``forget'' about the time spent in this layer.
The last two equations mean that if a part of the path is infeasible, then the whole path is infeasible.
The remaining ones have obvious interpretations.

% The first equation means that staying in layer for a duration $a$, followed by a stay of duration $b$ yields a stay of duration $a+b$.
% The second means that staying in a layer for a duration $a$ , and then staying in the same layer for a duration $x$ before performing a change of layer and arriving in a new layer with a minimum duration of $y$ is the same as staying in a first layer for a duration $a + x$, and then arriving in a new layer with a minimum duration of $y$.

%interpretation ? only for C(x, y) + C(u, v) : sequence of two change of layer can happens only if the min
%                                              duration is respected

\begin{prop}
The set $\Stay \cup \Change \cup \{ \infeas \}$, equipped with the operator
$\splus$ and with $\stay(0)$ as a neutral element forms a monoid.
\end{prop}
\begin{proof}
It is obvious from the definition of $\splus$ that $\stay(0)$ is a neutral element.
It remains to check the associativity of $\splus$. Several cases have to be checked.
In particular, one can verify that for any $a,b,x,y,u,v \in \R_+$ , we have 
\begin{align*}
& (\stay(a) \splus \change(x,y) ) \splus \stay(b) = \stay(a) \splus (\change(x,y)  \splus \stay(b)) = \change(a+x, y+b), \\
& (\change(a,b) \splus \change(u,v)) \splus \change(x,y) = \change(a,b) \splus (\change(u,v) \splus \change(x,y) ) =  \left\{\begin{array}{ll}
 \infeas & 
 \begin{array}{ll}
 \text{if } b + u < \tau_{\mathrm{lay}} \\ \text{or } v + x <  \tau_{\mathrm{lay}} 
 \end{array}, \\
 \change(a,y) & \text{otherwise,} \end{array}\right. \\
& (\change(u,v) \splus \stay(a)) \splus \change(x,y) = \change(u,v) \splus (\stay(a) \splus \change(x,y) ) =  \left\{\begin{array}{ll}
    \infeas & \text{if } v + a + x < \tau_{\mathrm{lay}}, \\
 \change(u,y) & \text{otherwise}.
 \end{array}\right.
\end{align*}
The other cases are trivial to check.
\end{proof}

% an example of what happens ?

We endow our monoid $\sset^{\mathrm{lay}}$ with the order $\sleq$ defined as follows.
% \todo{PB : why add the $\nsleq$ in the definition ? }
\begin{subequations}\label{eq:partialOrder}
  \begin{alignat}{3}
  \stay(a) &\sleq \stay(b) &\enskip&  \text{if } a \geq b &\enskip& \text{for all }a,b\in \R_+ \label{eq:sleqstst}\\
   \stay(a) &\sleq \change(x,y)  &\enskip&  \text{if } a \geq x \text{ and } a \geq y  &\enskip& \text{for all }a,x,y \in \R_+ \label{eq:sleqstch}\\
   \change(x,y)&\nsleq \stay(a) &\enskip&  &\enskip& \text{for all }x,y,a \in \R_+ \label{eq:sleqchst}\\
   \change(x,y)&\sleq \change(u,v)  &\enskip&  \text{if } x \geq u \text{ and }y\geq v  &\enskip& \text{for all } x,y,u,v \in \R_+ \\
   \stay(a)&\sleq \infeas  &\enskip&  \text{and }\infeas \nsleq \stay(a)  &\enskip& \text{for all }a \in \R_+ \label{eq:sleqstinfeas} \\
   \change(x,y) &\sleq \infeas  &\enskip&  \text{and } \infeas \nsleq \change(x,y) &\enskip& \text{for all }x,y \in \R_+ 
  \end{alignat}
\end{subequations}

\begin{prop}\label{prop:monoidSl}
$(\sset^{\mathrm{lay}},\splus,\sleq)$ is a lattice ordered monoid with meet operator $\smeet$ given by \eqref{eq:meet}.
\end{prop}
% \vpspace{-}
\begin{subequations}\label{eq:meet}
\begin{alignat}{2}
\stay(a) \smeet  \stay(b)  &= \stay(\max(a, b))        &\enskip&\text{ for all } a, b \in \R_+ \label{eq:meetstst}\\
\stay(a) \smeet  \change(x, y)=\change(x,y) \smeet  \stay(a)  &= \stay(\max(a,x,y)) &\enskip&\text{ for all } a, x, y \in \R_+\label{eq:meetstch}\\
\change(x, y) \smeet  \change(u, v) &=  \change(\max(x,u),\max(y,v))  &\enskip&\text{ for all }  x, y, u, v \in \R_+ \label{eq:meetchch} \\
% &= \change(x, y - a) &\enskip&\text{ for all } a, x, y \in \R_+\\
\infeas \smeet  \stay(a) &= \stay(a) \smeet  \infeas = \stay(a) &\enskip&\text{ for all } a \in \R_+ \label{eq:meatstinfeas}\\
\change(x,y) \smeet  \infeas &= \infeas \smeet  \change(x,y) = \change(x,y) &\enskip&\text{ for all }  x, y \in \R_+ \label{eq:meetchinfeas}
\end{alignat}
\end{subequations}
The join is defined and proved similarly, but it is not needed by the \MRCSP algorithms, we do not detail on the topic. 
% \todo{PB : it also needs a join for being a lattice, right ?}
\begin{proof}
We start by proving that $(\sset^{\mathrm{lay}},\sleq)$ is a lattice with meet operator $\smeet$. The definition of $\smeet$ in \eqref{eq:meetstst}, \eqref{eq:meetchch}, \eqref{eq:meatstinfeas}, and \eqref{eq:meetchinfeas} are trivial. We now consider \eqref{eq:meetstch}. Given $a,x$, and $y$ in $\R_+$, let $z=\stay(\max(a,x,y))$.
Equations \eqref{eq:sleqchst} and \eqref{eq:sleqstinfeas} implies that a lower bound on $\stay(a)$ is necessarily in $\Stay$. 
Equation \eqref{eq:sleqstst} and \eqref{eq:sleqstch} then ensures that a lower bound on $\stay(a)$ and $\change(x,y)$ is of the form $\stay(b)$ with $b \geq a$ and $b \geq \max(x,y)$, which gives $z = \stay(a) \smeet \change(x,y)$.
The join is defined similarly.
% implies $z \sleq \stay(a)$ and $\stay(b) \sleq z$ for any $\stay(b) \sleq \stay(a)$. Equation \eqref{eq:sleqstch} implies 

It remains to check the compatibility of $\splus$ and $\sleq$. There are several cases to check.
We detail only the cases where \eqref{eq:infeasibleSum} is involved, all the other ones being trivial. 
Let $a,x,y,z,t,u$, and $v$ be in $\R_+$.
We start by proving that $\stay(a) \sleq \change(x,y)$ implies $\stay(a) \splus
\change(u,v) \sleq \change(x,y) \splus \change(u,v)$.
Indeed, we have $\change(a+u,v) \sleq \infty$, and
as $a \geq x$ and $u\geq 0$, we have $a+u \geq x$, hence $\change(a+u,v) \sleq \change(x,v)$, 
% and $\change(a+u,v) \sleq \infty$, 
and the result follows from \eqref{eq:monoidSum} and \eqref{eq:partialOrder}. 
Similarly, $\change(u,v+a) \sleq \infeas$  and $\change(u,v+a) \sleq \change(u,y)$ provides
$\change(u,v) \splus \stay(a) \sleq \change(u,v) \splus \change(x,y)$.
Suppose now $\change(x,y) \sleq \change(z,t)$. Then $x\geq z$ and $y+u \geq t + u$ give $\change(x,y) \splus \change(u,v) \sleq \change(z,t)+ \change(u,v)$. Symmetrically, $v + x \geq v+ z$ and $y\geq t$ gives $\change(u,v) \splus \change(x,y) \sleq \change(u,v) \splus \change(z,t)$.
This concludes the proof.
% Then if $\stay(a) + \change(u,v) = \infeas$, then $y +u \leq a +u< \tau_{\mathrm{lay}}$ and $\change(x,y) + \change(u,v) = \infeas$.
% If $\stay(a) + \change(u,v) \neq \infeas$, then 
%  which gives $\stay(a) + \change(u,v) \sleq \change(x,y) + \change(u,v)$.
\end{proof}
% definition + interpretation of the order

% compatibility addition/order

% definition of meet and join
The mapping $\rmeas^{\mathrm{lay}} : \sset^{\mathrm{lay}} \rightarrow \{0,1\}$ defined as follows is non-decreasing with respect to $\sleq$.
\begin{equation}\label{eq:rmeasS}
   \rmeas^{\mathrm{lay}}(\infeas) = 1 \quad \text{and} \quad \rmeas^{\mathrm{lay}}(\stay(a)) = \rmeas^{\mathrm{lay}}(\change(x,y)) = 0 \text{ for all } a,x,y \in \R_+.
 \end{equation} 
 % Then $\rmeas^{\mathrm{lay}}$ is non-decreasing with respect to $\sleq$.

\begin{prop}\label{prop:SlencodesConstraint}
Let $p$ be a production plan,  $P$ be the corresponding $s$-$t$ path in $D$,
and $q_P^{\mathrm{lay}} \in \sset^{\mathrm{lay}}$ be its resource. Then $p$
satisfies the minimum duration constraint \emph{\ref{rule:min_layer_duration}} if and only if $\rmeas^{\mathrm{lay}}(\re_P^{\mathrm{lay}}) = 0$.
% and $q_P $ is an element of the form $\change(0, \tau)$, where $\tau$ is the time to be spent in the current layer before being allowed to change layer, or $q_p = \infeas$ and $P$ does not satisfy the minimum duration constraint.
\end{prop}
\begin{proof}
   A simple induction on the number of arcs in a path enables to prove that exactly one of the following is true :  
   %\begin{inparaenum}
    \begin{itemize}
    \item $q_P^{\mathrm{lay}}$ is equal to $\stay(a)$ and the corresponding production plan 
    contains no layer change. In this case, $a$ is the total duration spent in the layer;
    \item $q_P^{\mathrm{lay}}$ is equal to $\infeas$ and the plan contains two consecutive layer 
    changes such that the duration between the two changes is non-greater than $\tau_{\mathrm{lay}}$;
    \item $q_P^{\mathrm{lay}}$ is equal to $\change(x,y)$ and the plant contains at least one layer change. In this case, $x$ is the time spent in the first layer of the production plan, and $y$ the time spent in the last layer.
    \end{itemize}
   %\end{inparaenum}
   The result follows.
  % \todo{PB : a small remark : with our definition of the graph and monoid, a path starting at $s$ cannot be decorated with $\stay(a)$ : the first arc is always decorated with $\change(0, y)$, and then we add (on the right) either $\stay(a)$ of $\change(x,y)$. This procedure never yields $\stay(a)$ }
\end{proof}
% ordre :
%    C(x,y) <= C(u,v) si y <= v et C(x,y) <= infty
%    C(x,y) <= infty

% prop :
%    ordre est compatible avec l'operation

% remarque : pas besoin d'ordre sur les S(x), car sur le long d'un chemin, on n'aura que des C(x,y)
% compatibilite ordre/addition, Ok
% element neutre : S(0)

%%%%%%%%%%%%%%%%%%%%%%%%%%%%%%%%%%%%%%%%%%%%%%%%%%%%%%%%%%%%%%%%%%%%%%%%%%%%%%%
\subsection{A monoid for minimum mode duration constraints}
\label{sub:a_monoid_for_minimum_mode_duration_constraints}
%%%%%%%%%%%%%%%%%%%%%%%%%%%%%%%%%%%%%%%%%%%%%%%%%%%%%%%%%%%%%%%%%%%%%%%%%%%%%%%
 To model the minimum mode duration constraint, we use the lattice ordered
 monoid, \mbox{$(\sset^{\mathrm{mod}}, \splus, \sleq)$}, which is defined as
 $(\sset^{\mathrm{lay}}, \splus, \sleq)$, the only difference being that
 $\tau_{\mathrm{lay}}$ is replaced by $\tau_{\mathrm{mod}}$ in
 Equations~\eqref{eq:infeasibleSum} and \eqref{eq:rmeasS}. Given an arc $a$ and
 the corresponding transition $\alpha$, the resource of $a$ is $\change(0,0)$
 if $a$ is a mode change, and $\stay(\tau_\alpha)$ otherwise.
 The resource of a vertex starting in $s$ or ending in $t$ is $\change(0,0)$.
The two following propositions are proved like Propositions~\ref{prop:monoidSl} and~\ref{prop:SlencodesConstraint}.

\begin{prop}\label{prop:moinoidSm}
$(\sset^{\mathrm{mod}}, \splus, \sleq)$ is a lattice ordered monoid.
\end{prop}

\begin{prop}\label{prop:SmEncodesConstraint}
Let $p$ be a production plan, let $P$ be the corresponding $s$-$t$ path in $D$,
and $q_P^{\mathrm{mod}}$ be its resource in $\sset$. Then $p$ satisfies the
minimum duration constraint \emph{\ref{rule:min_mode_duration}} if and only if
$\rmeas^{\mathrm{mod}}(\re_P^\mathrm{mod}) = 0$.
\end{prop}

%%%%%%%%%%%%%%%%%%%%%%%%%%%%%%%%%%%%%%%%%%%%%%%%%%%%%%%%%%%%%%%%%%%%%%%%%%%%%%%
\subsection{Full monoid}
\label{sub:full_monoid}
%%%%%%%%%%%%%%%%%%%%%%%%%%%%%%%%%%%%%%%%%%%%%%%%%%%%%%%%%%%%%%%%%%%%%%%%%%%%%%%
We can now reduce a single unit commitment problem to a \MRCSP on $D$ with
resources in $\sset = \sset^{\mathrm{lay}} \cup \sset^{\mathrm{mod}} \cup
\Z_+^3 \cup \R$. We endow $(\rset,\rplus,\rleq)$ with the componentwise sum
and order, the sum and order on $\sset^{\mathrm{lay}} $ and
$\sset^{\mathrm{mod}} $ being those defined in the previous sections, and the
sum and order on $\Z_+$ and $\R$ being the standard ones. $(\rset,\rplus,\rleq)$ is a lattice ordered monoid as a product of lattice ordered monoids.

The resource of an arc is $(q^{\mathrm{lay}}, q^{\mathrm{mod}},n^{\mathrm{s}},
n^{\mathrm{l}},n^{\mathrm{d}},\tilde{c})$, where $q^{\mathrm{mod}}$ and $
q^{\mathrm{lay}}$ are defined as in the previous sections. If $a$ has $s$ as
tail or $t$ as head, then $n^{\mathrm{s}}=
n^{\mathrm{l}}=n^{\mathrm{d}}=\tilde{c} = 0$. Otherwise, $\alpha$ denotes the
transition corresponding to $a$, we choose $n^{\mathrm{s}} =
n_{\mathrm{startup}}(\alpha)$, $n^{\mathrm{l}} = n_{\mathrm{layer}}(\alpha)$,
and $n^{d} = n_{\mathrm{deep}}(\alpha)$.
Finally, given an arc $a$, we define $\tilde{c}_a$ to be equal to $0$ if $a$
has $s$ as tail or $t$ as destination and to $\lambda_{\alpha,t}$ otherwise,
with $\alpha$ and $t$ being the transition and timestep corresponding to $a$.

We define $\rmeas : \rset \rightarrow \R$ and $\rcost : \rset \rightarrow \R$ as 
\begin{align*}
\rmeas\big((q^{\mathrm{lay}}, q^{\mathrm{mod}},n^{\mathrm{s}}, n^{\mathrm{l}},n^{\mathrm{d}},\tilde{c})\big) &= 
\max\big(
\rmeas^{\mathrm{lay}}(q^{\mathrm{lay}}),
\rmeas^{\mathrm{mod}}(q^{\mathrm{mod}}),
\ind_{(n^{\mathrm{max}}_{\mathsf{startup}},\infty)}(n^{\mathrm{s}}), 
\\&\quad\quad\quad\quad\quad
\ind_{(n^{\mathrm{max}}_{\mathsf{layer}},\infty)}(n^{\mathrm{l}}),
\ind_{(n^{\mathrm{max}}_{\mathsf{deep}},\infty)}(n^{\mathrm{d}})\big),\\
\rcost\big((q^{\mathrm{lay}}, q^{\mathrm{mod}},n^{\mathrm{s}}, n^{\mathrm{l}},n^{\mathrm{d}},\tilde{c})\big) &= \tilde{c},
\end{align*}
where $\ind_{I}$ denotes the indicator function of interval $I$. The following proposition concludes the reduction of the single unit commitment problem to a \MRCSP.

\begin{prop}
Let $p$ be a production plan, $P$ be the corresponding $s$-$t$ path in $D$, and $q_P$ be its resource in $\rset$. Then $p$ is feasible if and only if $\rmeas(q_P) = 0$, and its cost is $c(q_P)$.
\end{prop}
\begin{proof}
Let $p$ be a production plan $(s_0,t_0, \alpha_0), (s_1, t_1, \alpha_1), \ldots, (s_k, t_k, \alpha_k)$, let $P$ be the corresponding $s$-$t$ path in $D$, and $q_P = (q_P^{\mathrm{l}}, q_P^{\mathrm{m}},n_P^{\mathrm{s}}, n_P^{\mathrm{l}},n_P^{\mathrm{d}},\tilde{c}_P)$ be its resource.
Propositions~\ref{prop:SlencodesConstraint} and~\ref{prop:SmEncodesConstraint} that $p$ satisfies constraints~\ref{rule:min_layer_duration} and~\ref{rule:min_mode_duration} respectively if and only if $\rmeas^{\mathrm{l}}(q_P^{\mathrm{l}}) = 0$ and $\rmeas^{\mathrm{m}}(q_P^{\mathrm{m}}) = 0$.
Furthermore, the definition of arcs resources ensures that $n_P^{\mathrm{s}}$, $n_P^{\mathrm{l}}$, and $n_P^{\mathrm{d}}$ respectively correspond to the number of startups, of layer change, and of deep transition in $p$.
The definition of $\rmeas$ then enables to conclude that $\rmeas(q_P) = 0$ if
and only if constraints~\ref{rule:min_layer_duration}--\ref{rule:max_nb_of_deep_transitions} are satisfied. 
Finally, the definition of arcs resources $c_a$ ensures that $\rcost(q_P) = \sum_{i = 1, \dots, k} \lambda(\alpha_i, t_i)$, which concludes the proof.
\end{proof}

% The proof of this proposition is i

\section{Numerical results} % (fold)
\label{sub:numericcal_results_single}

% The main results of the tested algorithms are shown in
\Cref{tab:algocomparision} provides numerical results obtained on a dataset of 97
independent thermal units of EDF. The columns of \Cref{tab:algocomparision} correspond to different algorithms.
The first three rows describe basic characteristics of the algorithms. They
are followed by four lines of algorithm performance statistics: the number of
path concatenation operations $\oplus$ done along the algorithm,
the number of partial paths discarded by the dominance and the lower bound
tests, and the number of $o$-$d$ paths returned by the algorithm. Finally, the
last lines give the total computation time, the proportion needed to compute
bounds, and the comparison to the standard
label setting algorithm \texttt{V0}, which does not
require lower bounds and use only dominance to discard paths. The smaller the ratio is, the better the algorithm is. The label
correcting algorithms correspond to the algorithm of Section
\ref{sec:ordered_monoid_framework}, and label setting
algorithms are described in \Cref{rem:LabelSetting}. Only algorithms
\texttt{V1} and \texttt{V3} use bounds and bounds computing time is included in
their respective solving time.

% The first label setting algorithm is the standard algorithm is the literaturen which does not use the lower bound test, the second uses

The label setting algorithms are distinguished by whether or not they use the lower bound
test to discard partial paths. It turns out that the usage of bounds lead to a
minor speed up of label setting algorithms.
% The overall
% potential of bounds in the context of label setting algorithms with the used
% key is limited, because most non promising partial paths are touched only for the reason
% that they end up in an early node. \\
This conclusion changes dramatically when it comes to label correcting
algorithms. Indeed, algorithm \texttt{V2}, which does not use bounds, performs
poorly, whereas algorithm \texttt{V3} is 25 times faster than
the usual label setting algorithms. The reason for this good performance is
that bounds enable to strongly reduce the number of partial paths explored, as
it can be observed on Figure~\ref{fig:search_tree}, which depicts the arcs
explored respectively by algorithms \texttt{V0} and \texttt{V3}. The use of
bounds in the key and in the test are both crucial to obtain this good
performance. Indeed, using bounds only in the test or only in the key leads to
algorithms that are respectively 45 and 20850 times slower than algorithm
\texttt{V3}. Extensive numerical results with these variants of the algorithms
are available in appendix.

% Compared to this, the second category of algorithms are label correcting
% algorithms \citep{irnich2005shortest}. The results of algorithm \texttt{V2} and \texttt{V3} show that
% only changing the selection strategy is not sufficient. Responsible for the huge speed up
% is the ability to discard paths early and more important, to identify promising paths by
% using an appropriate key. The lack of good bounds and their proper utilization
% might be the reason that label setting algorithms are more popular in the
% literature than the label correcting version. \\
% For the purpose of completeness: all variants without dominance test or the
% usage of bounds in either the key or the discard test lead to slower
% computation times compared to \texttt{V0} (\Cref{tab:algocomparisionComplete}).
% To provide a visual intuition, \Cref{fig:search_tree} shows the search space of
% \texttt{V0} compared to \texttt{V3} of an arbitrary unit. Under the given assumption that solving the RCSPP is part of a larger lagrangian
% relaxation approach, the improvement of a factor of $\sim$25 gains much more importance.

\begin{table}
\begin{center}
\begin{tabular}{r|cc|cc}
  & \texttt{V0} & \texttt{V1} & \texttt{V2} & \texttt{V3} \\
  \hline
  type & label setting & label setting & label correcting & label correcting \\
  \textsf{key} & early date & early date & $\rcost(q_p)$ & $\rcost(q_p \oplus b_v)$ \\
  \textsf{test} & Dom & Dom \& Low & Dom & Dom \& Low\\
  \hline
  iterations (k) & 326 & 281 & 150541 & 21 \\
  discarded Dom (k)  & 188 & 136 & 124083 & 8 \\
  discarded Low (k) & - & 31 & - & 8 \\
  \# od paths & 322 & 17 & 12342 & 1 \\
  \hline
  solving time (ms) & 40.3 & 36.8 & 96.6e3 & 1.8 \\
  bound computation ($\%$) & 0 & 0.4 & 0 & 8.3 \\
  solving time ratio & 1.00x & 0.91x & 2398x & 0.04x

\end{tabular}
% \begin{tabular}{r|cc|cc}
%   & \texttt{V0} & \texttt{V1} & \texttt{V2} & \texttt{V3} \\
%   \hline
%   type & label setting & label setting & label correcting & label correcting \\
%   key & early date & early date & $\rcost(q_p)$ & $\rcost(q_p \oplus b_v)$ \\
%   test & Dom & Dom \& Low & Dom & Dom \& Low\\
%   \hline
%   iterations (k) & 326.170 & 281.495 & 150.541.868 & 21.041  \\
%   discarded Dom (k)  & 188.366 & 136.606 & 124.083.361 & 8.946 \\
%   discarded Low (k) & 0 & 31.408 & 0 & 8.637 \\
%   \# od paths & 322 & 17 & 12.342 & 1 \\
%   \hline
%   solving time (ms) & 40.3 & 36.8 & 96.6e3 & 1.8 \\
%   bound computation ($\%$) & 0 & 0.4 & 0 & 8.3 \\
%   solving time ratio & 1.00x & 0.91x & 2398x & 0.04x
%
% \end{tabular}
\end{center}
\caption{An overview of implemented enumeration algorithms with different
  selection strategies, keys and discarding tests in use.}
\label{tab:algocomparision}
\end{table}

\begin{figure}[htb]
    \centering
    \begin{minipage}[t]{0.48\linewidth}
        \centering
        \includegraphics[width=\linewidth]{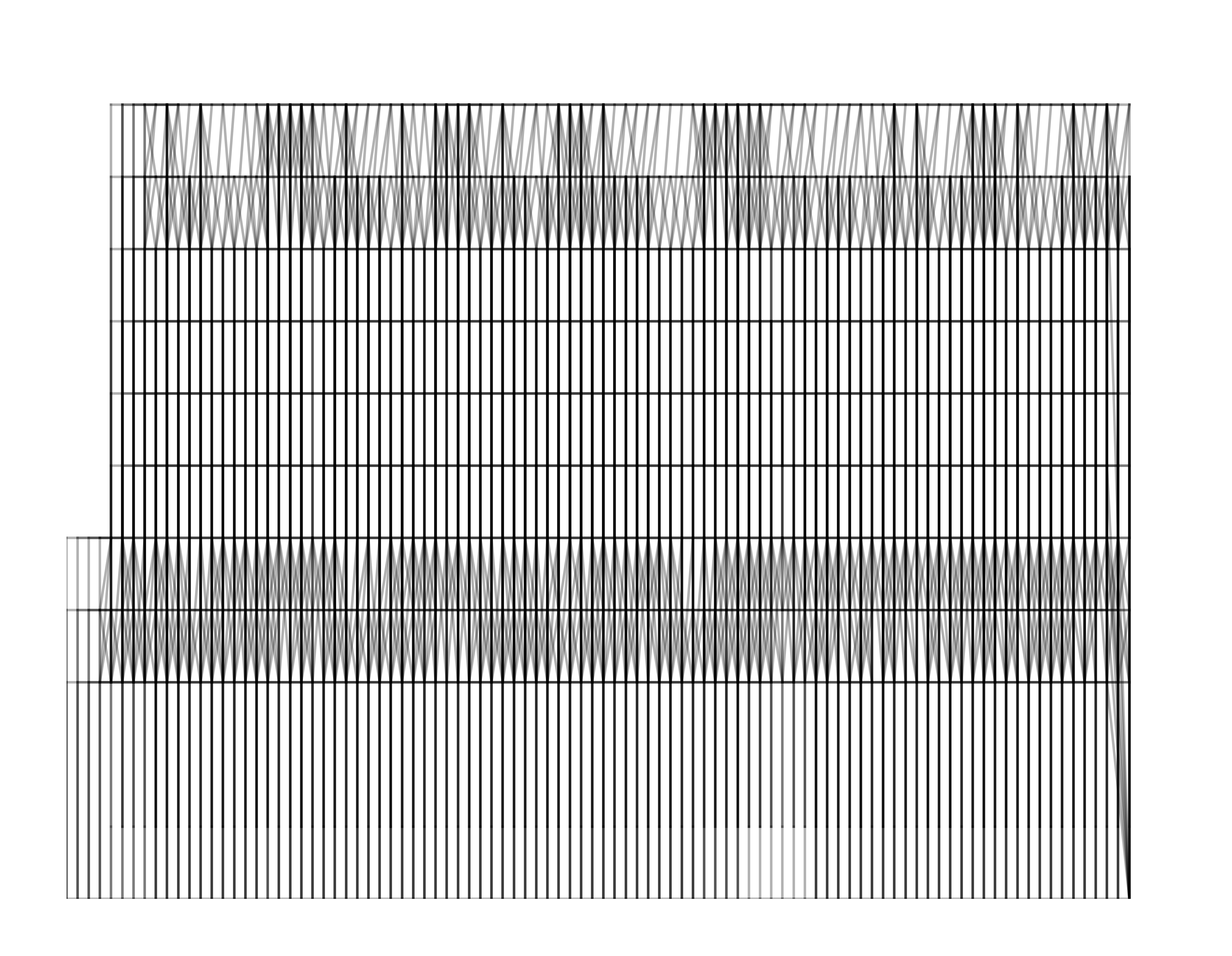}
        \caption*{a) Label setting algorithm \texttt{V0}.}
    \end{minipage}
    \hfill
    \begin{minipage}[t]{0.48\linewidth}
        \centering
        \includegraphics[width=\linewidth]{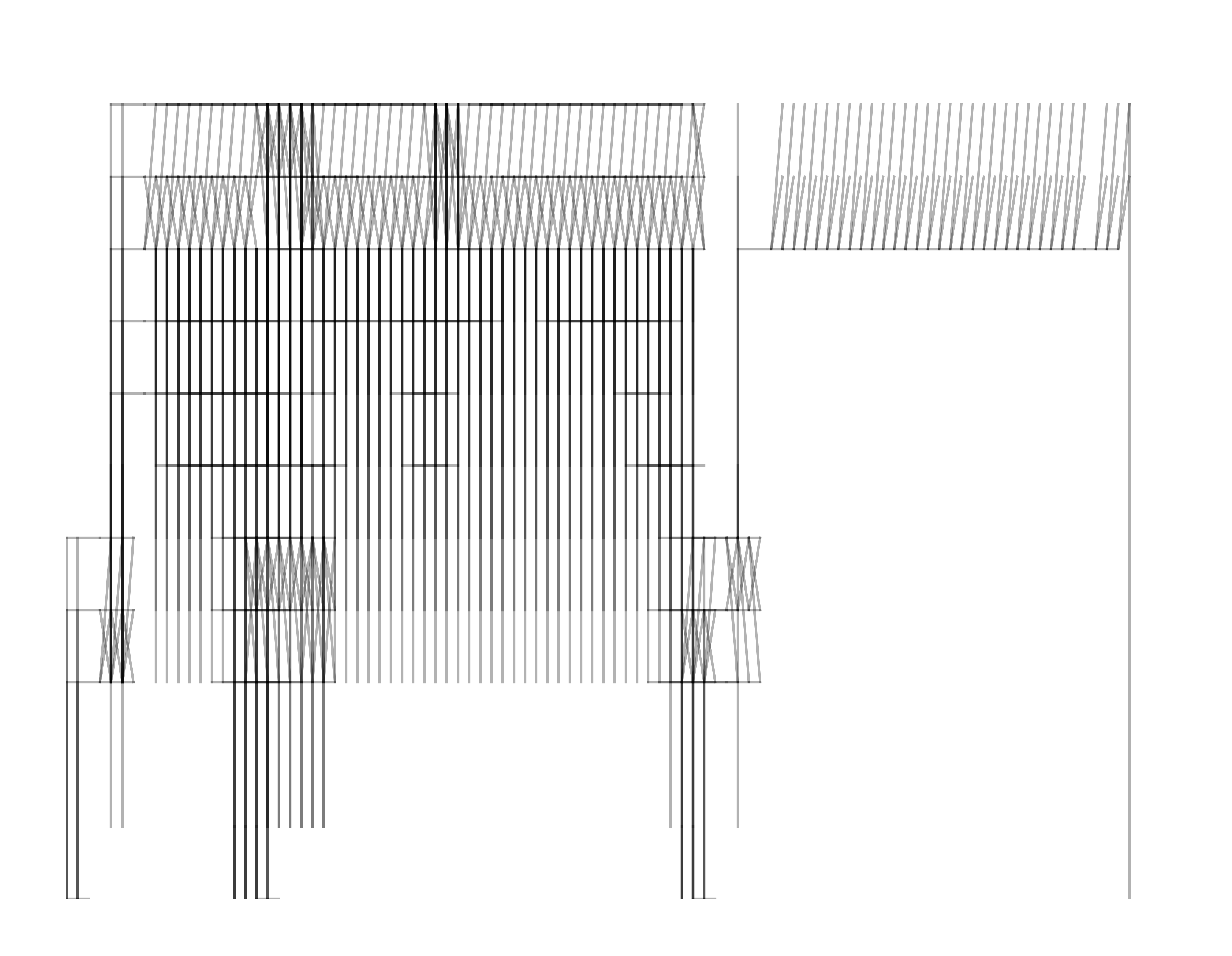}
        \caption*{b) Label correction algorithm \texttt{V3}.}
    \end{minipage}
  \caption{Visualization of the explored search space of digraph $D$
    % (\cref{sub:problem_statement})
    by the corresponding enumeration algorithm.
    The y-axis indicated power levels, and the x-axis the time horizon.
  %   Different power levels are located on
  % the y-axis (offline level at the bottom) and time on the x-axis.
  Each gray line represents a partial path during
  the solving process. Black lines indicate that multiple paths are using the same edge.}
\label{fig:search_tree}
\end{figure}

 % subsection numericcal_results (end)

% \red{Define the cost and feasibility function}

% \red{proposition: planning and arc resources correspond -- everything well defined}

%\todo{Only once the formula: where $t'=$ if, $=$ if, and $=$ otherwise.},

% subsection modeling_as_a_path_problem (end)

% A unit must follow procedures to change of level we model them by a set $\calA \subsetneq S^{2}\times \Z_+$ of \emph{level changes}. Given a level change $\alpha = (s_\alpha^{\mathrm{i}},s_\alpha^{\mathrm{f}},\tau_\alpha)$,  where in a subset $\calA$ of . Let $\calA$ be the set of level changes.

% section single_unit_commitment_problem (end)